\documentclass[preprint,12pt]{elsarticle}
\usepackage{amsfonts}
\usepackage{bbm}
\usepackage{mathrsfs}
\usepackage{mathrsfs}
\usepackage{amsfonts,mathrsfs,latexsym,amsmath,amssymb,amsthm}

\usepackage{amssymb}
\usepackage{bbding}

\makeatletter
 \@addtoreset{equation}{section}
 
\makeatother
\journal{Advances in Mathematics (China)}

\newcommand{\Z}{\mathbb{Z}}
\newcommand{\R}{\mathbb{R}}
\newtheorem{df}{Definition}[section]
\newtheorem{pr}[df]{Proposition}
\newtheorem{thm}[df]{Theorem}
\newtheorem{lm}[df]{Lemma}

\begin{document}

\begin{frontmatter}


\title{On the Standard Lattices\tnoteref{label1}}
\tnotetext[label1]{Project supported by NSFC (Grant No. 61370187) and by NSFC-Genertec
Joint Fund For Basic Research (Grant No. U1636104)}

\cortext[cor1]{Corresponding author}
\author[focal]{Rongquan Feng\corref{cor1}}
\ead{fengrq@math.pku.edu.cn}

\author[rvt]{Longke Tang}
\ead{tanglongke@pku.edu.cn}

\author[rvt]{Kun Wang}
\ead{wangkunmath@163.com}

\address[focal]{LMAM, School of Mathematical Sciences, Peking University, Beijing 100871, China}
\address[rvt]{School of Mathematical Sciences, Peking University, Beijing 100871, China}

\date{}


\begin{abstract}
A lattice in the Euclidean space is standard if it has a
basis consisting vectors whose norms equal to the length in its successive minima. In this paper, it is shown that with the $L^2$ norm all lattices of dimension $n$ are
standard if and only if $n\leqslant 4$. It is also proved that with an arbitrary norm, every lattice of dimensions 1 and 2 is standard.
An example of non-standard lattice of dimension $n\geqslant 3$ is given when the lattice is with the $L^1$ norm.

\vspace*{.5cm}
\end{abstract}


\end{frontmatter}



\section{Introduction}
A lattice is a discrete additive subgroup of a finitely dimensional $\R$-vector space, and geometry of
numbers is the theory that occupies itself with lattices. Since the publication of
Hermann Minkowski's {\em Geometrie der Zahlen} in 1896 (\cite{Mi}), lattices have become a
standard tool in number theory, especially in the areas of diophantine approximation,
algebraic number theory, and the arithmetic theory of quadratic forms (\cite{L}). Despite
their apparent simplicity, lattices hide a rich combinatorial structure,
which has attracted the attention of great mathematicians over the
last two centuries. Now lattices have found numerous applications not only
in mathematics, but also in computer science and in cryptography (\cite{MG}, \cite{SCB}, \cite{Si}).

Take $m$ linearly independent vectors in an $n$-dimensional Euclidean
space $\R^n$ $(m\leqslant n)$. The set of all $\Z$-linear
combinations of (or the abelian group generated by) these vectors is
called a lattice in $\R^n$. If $m=n$, the lattice is
said to be of full-rank. All lattices in this article are of
full-rank, and if not specified, $\R^n$ is with the Euclidean
metric, i.e., the ordinary $L^2$ norm. More precisely, we have the following definition:

\begin{df}
Let $\boldsymbol{b}_1,\ldots,\boldsymbol{b}_n$ be $n$ linearly
independent vectors in $\R^n$. The set
$$\Lambda=L(\boldsymbol{b}_1,\ldots,\boldsymbol{b}_n)=\left\{\sum_{i=1}^nx_i\boldsymbol{b}_i\;\Bigg|\; x_i\in\Z,\;i=1,\ldots,n\right\}$$
is called the \textbf{lattice} generated by
$\boldsymbol{b}_1,\ldots,\boldsymbol{b}_n$, and
$\boldsymbol{b}_1,\ldots,\boldsymbol{b}_n$ is called a
\textbf{basis} of $\Lambda$.
\end{df}

Obviously, a lattice (as a set of points or vectors in a Euclidean
space) has many different bases. The following proposition comes
directly from the above definition.

\begin{pr}
Let $B=\{\boldsymbol{b}_1,\ldots,\boldsymbol{b}_n\}$ and
$C=\{\boldsymbol{c}_1,\ldots,\boldsymbol{c}_n\}$ be two sets of
linearly independent vectors in $\R^n$. Then they generate the same
lattice (i.e. $C$ is a basis of the lattice generated by $B$) if and
only if they can express each other $\Z$-linearly, i.e. there exist
$u_{ij},v_{ij}\in\Z$ $(1\leqslant i,j\leqslant n)$ such that
$$\boldsymbol{b}_i=\sum_{j=1}^nu_{ij}\boldsymbol{c}_j,\;\;\mbox{and}\;\;\boldsymbol{c}_i=\sum_{j=1}^nv_{ij}\boldsymbol{b}_j\quad(1\leqslant i\leqslant n).$$
\end{pr}

For a subset $\Lambda\subseteq\R^n$, in general it is not convenient to determine whether $\Lambda$ is a lattice by finding its basis.
Noticing the abelian group structure of $\R^n$, we have the
following theorem.

\begin{thm}(\S3.2, \cite{GL})
For a subset $\Lambda\subseteq\R^n$, $\Lambda$ is a lattice in
$\R^n$ if and only if $\Lambda$ is a discrete additive subgroup of
$\R^n$ and not contained in any $(n-1)$-dimensional subspace of
$\R^n$.
\end{thm}

For any vector $\boldsymbol{x}\in \R^n$, denote by $\|\boldsymbol{x}\|$ the length (norm) of $\boldsymbol{x}$. Let $N(r)=\{\boldsymbol{x}\in\R^n\mid\|\boldsymbol{x}\|\leqslant
r\}$ be the closed ball of radius $r$ centered at the origin.

\begin{df}
Let $\Lambda$ be a lattice in $\R^n$. For $i=1,\ldots,n$, let
$$\lambda_i=\inf\{r\mid\dim\,{\rm Span}(\Lambda\cap N(r))\geqslant i\},$$
where ${\rm Span}(X)$
denotes the subspace of $\R^n$ spanned by the set $X$. The sequence $\lambda_1,\ldots,\lambda_n$ is called the
\textbf{successive minima} of $\Lambda$.
\end{df}

It can be easily seen from the definition that the successive minima of a
lattice is unique.

\begin{df}
Let $\Lambda$ be a lattice in $\R^n$ generated by the basis
$\boldsymbol{b}_1,\ldots,\boldsymbol{b}_n$, and let
$\lambda_1,\ldots,\lambda_n$ be the successive minima of $\Lambda$.
If $$\|\boldsymbol{b}_i\|=\lambda_i,\quad \mbox{for all}\; 1\leqslant i\leqslant n,$$ then the basis
$\boldsymbol{b}_1,\ldots,\boldsymbol{b}_n$ is said to
\textbf{achieve} its successive minima.
\end{df}

By discreteness and the definition of the successive minima, we have the following theorem.

\begin{thm}(\S3.3, \cite{S})
For a lattice $\Lambda$ in $\R^n$ with the successive minima
$\lambda_1,\ldots,\lambda_n$, there exist $\R$-linearly
independent vectors $\boldsymbol{u}_1,\cdots,\boldsymbol{u}_n\in\Lambda$ satisfies
$\|\boldsymbol{u}_i\|=\lambda_i$ for all $i=1,\ldots,n$.
\end{thm}

Clearly the vectors $\boldsymbol{u}_1,\ldots,\boldsymbol{u}_n$ above
form an $\R$-basis of $\R^n$, but they might not form a basis of
$\Lambda$.

\begin{df}
A lattice in $\R^n$ is said to be \textbf{standard} if it has a
basis achieving its successive minima; otherwise, the lattice is
said to be \textbf{non-standard}.
\end{df}

Obviously the lattice $\Z^n\subseteq\R^n$ is standard. The following result can be checked easily.

\begin{pr}
If a lattice $\Lambda\subseteq\R^n$ has an orthogonal basis (i.e. a
basis which contains $n$ mutually orthogonal vectors), then
$\Lambda$ is standard, and this basis achieves the successive
minima.
\end{pr}


It is natural to discuss whether a given lattice is standard,
especially to determine the dimensions of which all lattices are
standard. By Theorem 1.6, a lattice $\Lambda$ in $\R^n$ contains $n$
linearly independent vectors achieving its successive
minima, and there is a sublattice
$\Lambda^{\prime}$ generated by them. By using the Hadamard inequality and
Minkowski's inequality, an estimate of the order of the
quotient group $\Lambda/\Lambda^{\prime}$ was given in \cite{M}. This gave the result that all lattices of
dimensions not greater than $3$ are standard. The case of dimension 4 was also discussed in \cite{M}.

In this paper,
we will give a direct geometric proof that with the $L^2$ norm all lattices of
dimension $n$ are standard if and only if $n\leqslant 4$ in Sections 2 and 3. A brief discussion on lattices in $\R^n$
of arbitrary norms is given in Section 4.

\section{High-dimensional Cases}

We first give an example of non-standard lattice when the dimension
is greater than 4.

\begin{thm}\label{ce}
For $n\geqslant5$, there exist non-standard lattices in $\R^n$.
\end{thm}

\begin{proof}
For $n\geqslant5$, let
$$\Lambda_n=\{(a_1,\ldots,a_n)\in\Z^n\mid a_1\equiv\cdots\equiv a_n\pmod{2}\};$$
then $\Lambda_n$ is a discrete additive subgroup of $\R^n$ not
contained in any $(n-1)$-dimensional subspace. By Theorem 1.3 it is
a lattice. In fact, it has $(2,0,\ldots,0)$, $(0,2,0,\ldots,0)$,
\ldots, $(0,\ldots,0,2,0)$, $(1,1,\ldots,1)$ as a basis.

On the one hand, note that for all nonzero $\boldsymbol{a}=(a_1,\ldots,a_n)\in\Lambda_n$, if
$$a_1\equiv\cdots\equiv a_n\equiv1\pmod{2},$$
then $a_i\ne0$ for all $i$, so
$$\|\boldsymbol{a}\|=\sqrt{\sum_{i=1}^na_i^2}\geqslant\sqrt{n}\geqslant\sqrt{5}>2;$$
if
$$a_1\equiv\cdots\equiv a_n\equiv0\pmod{2},$$
then there exists some $j\in\{1,2,\ldots,n\}$ with $a_j\ne0$, so
$$\|\boldsymbol{a}\|\geqslant|a_j|\geqslant2.$$
Note also that there are $n$ linearly independent vectors of length 2
in $\Lambda_n$, e.g. $2\boldsymbol{e}_1,\ldots,2\boldsymbol{e}_n$,
where $\boldsymbol{e}_i$ is the vector whose $i$-th coordinate is 1 and all other coordinates are 0's,
so the successive minima of $\Lambda_n$ is
$\lambda_1=\cdots=\lambda_n=2$.

On the other hand, note that every basis
$\boldsymbol{b}_1,\ldots,\boldsymbol{b}_n$ of $\Lambda_n$ must
contain a vector with odd coordinates. In fact, if
$\boldsymbol{b}_1,\ldots,\boldsymbol{b}_n$ all have even
coordinates, then $(1,1,\ldots,1)\in\Lambda_n$ cannot be expressed
as a $\Z$-linear combination of them, which is a contradiction.
Suppose that $\boldsymbol{b}_k$ has odd coordinates, by the previous
discussion, we have
$$\|\boldsymbol{b}_k\|\geqslant\sqrt{n}\geqslant\sqrt{5}>2;$$
therefore, no basis $\boldsymbol{b}_1,\ldots,\boldsymbol{b}_n$ of
$\Lambda_n$ can achieve the successive minima, which shows that
$\Lambda_n$ is non-standard.
\end{proof}

Note that van der Waerden \cite{vdW} gave the same counterexample in 1956 and it can also be found in \cite{MG}.

\section{Low-dimensional Cases}

In this section, we will prove that all lattices of dimension less than 5 are standard.

\begin{pr}\label{d1}
Every lattice in $\R^1$ is standard.
\end{pr}

\begin{proof}
Let $\boldsymbol{b}_1$ be a basis of a lattice
$\Lambda\subseteq\R^1$; then
$$\Lambda=L(\boldsymbol{b}_1)=\{k\boldsymbol{b}_1\mid k\in\Z\}.$$
Obviously,
$$\lambda_1=\inf\{\|k\boldsymbol{b}_1\|\mid k\in\Z\setminus0\}=\|\boldsymbol{b}_1\|,$$
so $\Lambda$ is standard.
\end{proof}

In order to discuss lattices of other dimensions, we give a lemma about the structure of a lattice in $\R^n$ firstly.

\begin{lm}\label{bd}
Let $\Lambda$ be a lattice in $\R^n$ generated by
$\boldsymbol{b}_1,\ldots,\boldsymbol{b}_n$, and let $\lambda>0$
satisfy
$$\|\boldsymbol{b}_i\|\leqslant\lambda\;\mbox{for all}\; 1\leqslant i\leqslant n.$$
Then for any $\boldsymbol{v}\in\R^n$, we have
$$\min_{\boldsymbol{u}\in\Lambda}\|\boldsymbol{v}-\boldsymbol{u}\|\leqslant\frac{\sqrt{n}}{2}\lambda;$$
(the minimum value exists because of the discreteness of $\Lambda$). Furthermore, the equality holds only when
$\boldsymbol{b}_1,\ldots,\boldsymbol{b}_n$ are mutually orthogonal,
$\|\boldsymbol{b}_i\|=\lambda$ for all $1\leqslant i\leqslant n$, and
$$\boldsymbol{v}=\sum_{i=1}^n\left(a_i+\frac{1}{2}\right)\boldsymbol{b}_i$$
for some $a_i\in\Z$, $1\leqslant i\leqslant n$.
\end{lm}

\begin{proof}
Use induction on $n$. For $n=1$ the lemma is obvious. Assume it is
true for dimension $n-1$. Since
$\boldsymbol{b}_1,\ldots,\boldsymbol{b}_n$ ia a basis of the
$\R$-linear vector space $\R^n$, we have that
$$\boldsymbol{v}=\sum_{i=1}^nc_i\boldsymbol{b}_i,$$
where $c_i\in\R$, $1\leqslant i\leqslant n$. Take $a_n\in\Z$ with $|a_n-c_n|\leqslant\frac{1}{2}$. Let
$\boldsymbol{v}'$ and $\boldsymbol{b}'_n$ be
the orthogonal projections of $\boldsymbol{v}$ and $\boldsymbol{b}_n$ respectively
onto the hyperplane $H$ spanned by
$\boldsymbol{b}_1,\ldots,\boldsymbol{b}_{n-1}$. The lattice
$L=L(\boldsymbol{b}_1,\ldots,\boldsymbol{b}_{n-1})$ generated by
$\boldsymbol{b}_1,\ldots,\boldsymbol{b}_{n-1}$ is a lattice in $H$
of dimension $n-1$. By the induction hypothesis, since
$\boldsymbol{v}'-a_n\boldsymbol{b}'_n\in H$, there exists a vector
$\boldsymbol{u}'\in L$ such that
$$\|\boldsymbol{v}'-a_n\boldsymbol{b}'_n-\boldsymbol{u}'\|=\min_{\boldsymbol{w}\in L}\|\boldsymbol{v}'-a_n\boldsymbol{b}'_n-\boldsymbol{w}\|\leqslant\frac{\sqrt{n-1}}{2}\lambda.$$
Projecting to the line orthogonal to the above hyperplane $H$, we
have
$$\boldsymbol{v}'=\sum_{i=1}^{n-1}c_i\boldsymbol{b}_i+c_n\boldsymbol{b}_n',$$
and then $\boldsymbol{v}-\boldsymbol{v}'=c_n(\boldsymbol{b}_n-\boldsymbol{b}'_n)$.
Since $(\boldsymbol{b}_n-\boldsymbol{b}'_n)\bot H$,
$\boldsymbol{b}'_n\in H$, by Pythagorean theorem we have
$$\|\boldsymbol{b}_n-\boldsymbol{b}'_n\|=\sqrt{\|\boldsymbol{b}_n\|^2-\|\boldsymbol{b}'_n\|^2}\leqslant|\boldsymbol{b}_n\|\leqslant\lambda.$$
Hence
$$\|\boldsymbol{v}-\boldsymbol{v}'-a_n(\boldsymbol{b}_n-\boldsymbol{b}'_n)\|=
|a_n-c_n|\cdot \|\boldsymbol{b}_n-\boldsymbol{b}'_n\|\leqslant
\frac{1}{2}\lambda.$$ Since
$\boldsymbol{v}'-a_n\boldsymbol{b}'_n-\boldsymbol{u}'\in H$,
$(\boldsymbol{v}-\boldsymbol{v}'-a_n(\boldsymbol{b}_n-\boldsymbol{b}'_n))\bot
H$, again by Pythagorean theorem, we have
$$\|\boldsymbol{v}-a_n\boldsymbol{b}_n-\boldsymbol{u}'\|=
\sqrt{\|\boldsymbol{v}-\boldsymbol{v}'-a_n(\boldsymbol{b}_n-\boldsymbol{b}'_n)\|^2+
\|\boldsymbol{v}'-a_n\boldsymbol{b}'_n-\boldsymbol{u}'\|^2}\leqslant\frac{\sqrt{n}}{2}\lambda,$$
which completes the proof of the inequality, i.e., $\boldsymbol{u}_0=a_n\boldsymbol{b}_n+\boldsymbol{u}'\in \Lambda$ satisfies $\|\boldsymbol{v}-\boldsymbol{u}_0\|\leqslant\frac{\sqrt{n}}{2}\lambda$. Note that in order to achieve
the equality, we must have $$\|\boldsymbol{v}-\boldsymbol{u}_0\|=\min_{\boldsymbol{u}\in\Lambda}\|\boldsymbol{v}-\boldsymbol{u}\|,$$ and $\boldsymbol{b}_1,\ldots,\boldsymbol{b}_{n-1}$ have to be
mutually orthogonal with length $\lambda$ by the induction hypothesis,
$\|\boldsymbol{b}_n\|=\lambda$, and $\|\boldsymbol{b}'_n\|=0$, i.e.
$\boldsymbol{b}_n$ has to be orthogonal to the hyperplane spanned by
$\boldsymbol{b}_1,\ldots,\boldsymbol{b}_{n-1}$. In addition, $c_1,\ldots,c_{n-1}$ has to be half-integers by
the induction hypothesis, and $c_n$ is a half-integer since
$|a_n-c_n|=\frac{1}{2}$. It is clear that the equality holds under the above conditions, which completes the proof.
\end{proof}

\begin{thm}\label{mt}
For $n\leqslant4$, every lattice in $\R^n$ is standard.
\end{thm}

\begin{proof}
Use induction on $n$. We have proved the case of $n=1$ in Theorem
\ref{d1}. Assume that the theorem holds for dimension $n-1$ ($2\leqslant
n\leqslant4$). Denote the lattice by $\Lambda$ and its successive
minima by $\lambda_1,\ldots,\lambda_n$. By Theorem 1.6, there exists
an $\R$-basis
$\{\boldsymbol{c}_1,\ldots,\boldsymbol{c}_n\}\subseteq\Lambda$ of
$\R^n$ with $\|\boldsymbol{c}_i\|=\lambda_i$ for $1\leqslant i\leqslant n$. Let $H$ be
the hyperplane spanned by
$\boldsymbol{c}_1,\ldots,\boldsymbol{c}_{n-1}$, we can see that
$L=H\cap\Lambda$ is a discrete additive subgroup of $H$ and not
contained in any $(n-2)$-dimensional subspace of $H$, hence by
Theorem 1.3 $L$ is a lattice in the $(n-1)$-dimensional Euclidean
space $H$. By definition the successive minima of the lattice
$L\subseteq\Lambda$ is $\lambda_1,\ldots,\lambda_{n-1}$. By
the induction hypothesis, $L$ is standard, i.e. there exists a basis
$\boldsymbol{b}_1,\ldots,\boldsymbol{b}_{n-1}$ of $L$ with
$\|\boldsymbol{b}_i\|=\lambda_i,1\leqslant i\leqslant n-1$. Let
$\boldsymbol{b}_n=\boldsymbol{c}_n$, then $\boldsymbol{b}_n\notin
H$, which implies that
$\{\boldsymbol{b}_1,\ldots,\boldsymbol{b}_n\}\subseteq\Lambda$ is an
$\R$-basis of $\R^n$.

Let $K=L(\boldsymbol{b}_1,\ldots,\boldsymbol{b}_n)$. If $\Lambda$ is non-standard, for any
$\boldsymbol{v}\in\Lambda\setminus K$ (such $\boldsymbol{v}$ exists
since the non-standard $\Lambda$ is not generated by
$\boldsymbol{b}_1,\ldots,\boldsymbol{b}_n$), by Lemma 3.2, there
exists some $\boldsymbol{u}\in K$ such that
$$\|\boldsymbol{v}-\boldsymbol{u}\|=\min_{\boldsymbol{w}\in K}\|\boldsymbol{v}-\boldsymbol{w}\|\leqslant\frac{\sqrt{n}}{2}\lambda_n.$$
It is clear that (by the additive group structure of lattice)
$\boldsymbol{v}-\boldsymbol{u}\in\Lambda$,
$\boldsymbol{v}-\boldsymbol{u}\notin K$, and $L=H\cap\Lambda$ is contained in $K$, so
$\boldsymbol{v}-\boldsymbol{u}\notin H$. In other words,
$\boldsymbol{b}_1,\ldots,\boldsymbol{b}_{n-1},\boldsymbol{v}-\boldsymbol{u}$
is also an $\R$-basis of $\R^n$. Hence, by the definition of the
successive minima, we have
$$\lambda_n\leqslant\|\boldsymbol{v}-\boldsymbol{u}\|\leqslant\frac{\sqrt{n}}{2}\lambda_n.$$
(If $\lambda_n>\|\boldsymbol{v}-\boldsymbol{u}\|$, since $0\neq
\boldsymbol{v}-\boldsymbol{u}\in\Lambda$,
$\lambda_1\leqslant\|\boldsymbol{v}-\boldsymbol{u}\|$. Choose
$1\leqslant k<n$ such that
$\lambda_k\leqslant\|\boldsymbol{v}-\boldsymbol{u}\|<\lambda_{k+1}$,
then
$\boldsymbol{b}_1,\ldots,\boldsymbol{b}_k,\boldsymbol{v}-\boldsymbol{u}\in\Lambda$
are $k+1$ linearly independent vectors with length $<\lambda_{k+1}$, which is a contradiction.) Hence $n\geqslant4$.

If $n=2$ or $n=3$, the contradiction above shows that $\Lambda$ is
standard. If $n=4$, all inequalities above must be equalities. Now
by Lemma 3.2,
$$\min_{\boldsymbol{w}\in K}\|\boldsymbol{v}-\boldsymbol{w}\|=\frac{\sqrt{4}}{2}\lambda_4$$
implies that
$\boldsymbol{b}_1,\boldsymbol{b}_2,\boldsymbol{b}_3,\boldsymbol{b}_4$
are mutually orthogonal,
$\|\boldsymbol{b}_1\|=\|\boldsymbol{b}_2\|=\|\boldsymbol{b}_3\|=\|\boldsymbol{b}_4\|$,
i.e. $\lambda_{1}=\lambda_{2}=\lambda_{3}=\lambda_{4}$, and for any $\boldsymbol{v}\in\Lambda\setminus K$, one has
$\boldsymbol{v}=\sum_{i=1}^4(a_i+\frac{1}{2})\boldsymbol{b}_i$ for some $a_i\in\Z$, $i=1,\cdots,4$.
So
$$\Lambda\setminus K\subseteq\left\{\sum_{i=1}^{4}\left(d_{i}+\frac{1}{2}\right)\boldsymbol{b}_{i}\mid d_{i}\in\Z, i=1,\cdots,4\right\}.$$
Moreover, since there exists a vector
$\boldsymbol{v}=\sum_{i=1}^4(a_i+\frac{1}{2})\boldsymbol{b}_i\in\Lambda$, where $a_i\in\Z$, $i=1,\cdots,4$,
for any
$\boldsymbol{w}=\sum_{i=1}^4(d_i+\frac{1}{2})\boldsymbol{b}_i$,
$d_i\in\Z, i=1,\cdots,4$, we have
$$\boldsymbol{w}_0=\boldsymbol{w}-\sum_{i=1}^4(d_i-a_i)\boldsymbol{b}_i\in \Lambda,$$  hence $\boldsymbol{w}=\boldsymbol{w}_0+\boldsymbol{v}\in\Lambda$,
which implies that
$$\Lambda=K\cup\left\{\sum_{i=1}^{4}\left(a_{i}+\frac{1}{2}\right)\boldsymbol{b}_{i}\mid a_{i}\in\mathbb{Z}, i=1,\cdots,4\right\}.$$
It is clear that
$\boldsymbol{b}_1,\boldsymbol{b}_2,\boldsymbol{b}_3,\frac{1}{2}(\boldsymbol{b}_1+\boldsymbol{b}_2+\boldsymbol{b}_3+\boldsymbol{b}_4)$
is a basis of the lattice $\Lambda$, which achieve the successive
minima. Thus $\Lambda$ is standard, which completes the proof.
\end{proof}

\section{Arbitrary Norms}

Note that the definition of lattice depends only on the algebraic
structure of  $\R^n$ as a linear space, but the successive minima
depends on the norm. An interesting discussion on successive minima with respect to arbitrary
norms can be found in \cite{C}. In this section, we require no longer the
Euclidean metric in $\R^n$. We will see that in an arbitrary norm in
$\R^n$, every lattice is standard if and only if $n=1,2$. The case
of $n=1$ is trivial and the proof in Proposition 3.1 still applies.
The proof of next theorem shows the the algebraic structure of the
lattice in $\R^2$ deeply.

\begin{thm}
For an arbitrary norm in $\R^2$, every lattice is standard.
\end{thm}

\begin{proof}
Let $\Lambda$ denote the lattice in $\R^2$ and $\lambda_1,\lambda_2$
denote its successive minima. Let $\boldsymbol{b}_1\in\Lambda$ be
such that
$$\|\boldsymbol{b}_1\|=\min\{\|\boldsymbol{u}\|\mid\boldsymbol{u}\in\Lambda\},$$
the minimum value exists because of the discreteness of $\Lambda$.
Thus $\|\boldsymbol{b}_1\|=\lambda_1$. Let
$\boldsymbol{b}_2\in\Lambda$ be such that
$$\|\boldsymbol{b}_2\|=\min\{\|\boldsymbol{u}\|\mid\Lambda=L(\boldsymbol{b}_1,\boldsymbol{u})\},$$
at the end of the proof we will prove that such vector $\boldsymbol{u}$
exists, and the minimum value exists because of the discreteness of
$\Lambda$. Therefore $\boldsymbol{b}_1$ and
$\boldsymbol{b}_2$ generate $\Lambda$, and
$\|\boldsymbol{b}_2\|\geqslant\lambda_1$. Moreover,
$\boldsymbol{b}_1,\boldsymbol{b}_2$ are linearly independent with
norm $\leqslant\|\boldsymbol{b}_2\|$, hence
$\lambda_2\leqslant\|\boldsymbol{b}_2\|$. We need only to prove that
$\|\boldsymbol{b}_2\|=\lambda_2$, which shows $\Lambda$ has a basis
$\boldsymbol{b}_1,\boldsymbol{b}_2$ achieving its successive minima,
and thus $\Lambda$ is standard.

Assume otherwise, $\|\boldsymbol{b}_2\|>\lambda_2$. Then there exist two linearly independent vectors in $\Lambda$ with
norm $\lambda_1$ and $\lambda_2$ respectively. Thus at least one of these two linearly independent vectors is linearly independent with
$\boldsymbol{b}_1$. Denote this vector by $\boldsymbol{v}$ and write
$$\boldsymbol{v}=a_1\boldsymbol{b}_1+a_2\boldsymbol{b}_2\in\Lambda$$ with $a_1,a_2\in\Z$.
So $a_2\ne0$, and
$\|\boldsymbol{v}\|\leqslant\lambda_2<\|\boldsymbol{b}_2\|$, which
implies, by the definition of $\boldsymbol{b}_2$, that
$L(\boldsymbol{b}_1,\boldsymbol{v})\subsetneqq\Lambda$, and thus
$a_2\ne\pm1$ since
$L(\boldsymbol{b}_1,a_1\boldsymbol{b}_1\pm\boldsymbol{b}_2)=L(\boldsymbol{b}_1,\boldsymbol{b}_2)$
by Proposition 1.2. Therefore $|a_2|\geqslant2$.

Write $a_1=qa_2+r$ with $q,r\in\Z$, where $0\leqslant r<|a_2|$. Note
that by Proposition 1.2,
$L(\boldsymbol{b}_1,q\boldsymbol{b}_1+\boldsymbol{b}_2)=L(\boldsymbol{b}_1,\boldsymbol{b}_2)=\Lambda$,
so by the choice of $\boldsymbol{b}_2$,
$\|q\boldsymbol{b}_1+\boldsymbol{b}_2\|\geqslant\|\boldsymbol{b}_2\|$.
Hence
$$|a_2|\|\boldsymbol{b}_2\|\leqslant\|qa_2\boldsymbol{b}_1+a_2\boldsymbol{b}_2\|=
\|\boldsymbol{v}-r\boldsymbol{b}_1\|\leqslant\|\boldsymbol{v}\|+r\|\boldsymbol{b}_1\|<
(1+r)\|\boldsymbol{b}_2\|$$ since
$\|\boldsymbol{v}\|<\|\boldsymbol{b}_2\|$ and
$\|\boldsymbol{b}_1\|\leqslant\|\boldsymbol{b}_2\|$. Thus, $|a_2|<1+r$,
contradicting the fact that $0\leqslant r<|a_2|$. This completes the
proof.

Now we will prove that there exists $\boldsymbol{u}$ such
that $\Lambda=L(\boldsymbol{b}_1,\boldsymbol{u})$, i.e.
$\boldsymbol{b}_1,\boldsymbol{u}$ is a basis of the lattice
$\Lambda$. By Theorem 1.3, $\Lambda$ is an additive subgroup, which
can be regarded as a finitely generated free $\Z$-module of rank
$2$, and $L(\boldsymbol{b}_1)=\Z\boldsymbol{b}_1$ be its free
$\Z$-submodule of rank $1$. Then the quotient module
$\Lambda/L(\boldsymbol{b}_1)$ is finitely generated. By the
structure theorem of finitely generated module over principle ideal
domains (\S1.5, \cite{XZ}), $\Lambda/L(\boldsymbol{b}_1)$ is a direct sum
of a free $\Z$-module $F$ and a torsion $\Z$-module
$$T=\{\overline{\boldsymbol{w}}=\boldsymbol{w}+L(\boldsymbol{b}_1)\in\Lambda/L(\boldsymbol{b}_1)\mid\exists
m\in\Z\setminus\{0\},m\overline{\boldsymbol{w}}=\bar{\boldsymbol{0}}\},$$
i.e., $\Lambda/L(\boldsymbol{b}_1)=F\oplus T$. It is easy to see that
$T={\bar{\boldsymbol{0}}}$ is trivial, hence
$\Lambda/L(\boldsymbol{b}_1)=F$ is a free $\Z$-module. In fact, if
there exists $\overline{\boldsymbol{w}}\in
T\setminus\{\bar{\boldsymbol{0}}\}$, i.e.
$\boldsymbol{w}\in\Lambda\setminus L(\boldsymbol{b}_1)$, then there
exists an integer $m\neq0$ such that
$\overline{m\boldsymbol{w}}=m\overline{\boldsymbol{w}}=\bar{0}$, i.e.
$m\boldsymbol{w}\in L(\boldsymbol{b}_1)$. Suppose that
$m\boldsymbol{w}=l\boldsymbol{b}_1$, $l\in\Z$, then
$\boldsymbol{w}=\frac{l}{m}\boldsymbol{b}_1$. Since
$\boldsymbol{w}\notin L(\boldsymbol{b}_1)$, $\frac{l}{m}\notin\Z$.
Hence
$\boldsymbol{0}\neq\boldsymbol{w}-\lfloor\frac{l}{m}\rfloor\boldsymbol{b}_1\in\Lambda$,
and
$$0\neq\left\|\boldsymbol{w}-\left\lfloor\frac{l}{m}\right\rfloor\boldsymbol{b}_1\right\|
=\left\|\left(\frac{l}{m}-\left\lfloor\frac{l}{m}\right\rfloor\right)\boldsymbol{b}_1\right\|<\|\boldsymbol{b}_1\|,$$
where $\lfloor\frac{l}{m}\rfloor$ is the greatest integer not greater then
$\frac{l}{m}$. Contradicting the definition of $\boldsymbol{b}_1$.
So $T$ is trivial.

It is clear that $\Lambda/L(\boldsymbol{b}_1)$ is not trivial, and
choose a basis $\overline{\boldsymbol{x}_i},i=1,\cdots,r$ of the
free module $\Lambda/L(\boldsymbol{b}_1)$ of rank $r$, then
$\boldsymbol{x}_1,\cdots,\boldsymbol{x}_r,\boldsymbol{b}_1$ is a
basis of the free module $\Lambda$ of rank $2$. Hence $r+1=2$, and then
$r=1$. Thus $\boldsymbol{u}=\boldsymbol{x}_1$
satisfies $\Lambda=L(\boldsymbol{b}_1,\boldsymbol{u})$, as required.
\end{proof}

\begin{pr}
For $n\geqslant3$, there exists non-standard lattices in $\R^n$ with
$L^1$-norm.

\begin{proof}
We can show that the lattice $\Lambda_n$ constructed in the proof of Theorem
\ref{ce} are non-standard with $L^1$-norm for $n\geqslant3$.
Similar as the proof in Theorem 2.1, the successive minima of
$\Lambda_n$ is $\lambda_1=\cdots=\lambda_n=2$, but any vector in
$\Lambda_n$ with odd coordinates has norm $\geqslant n>2$. This
implies that $\Lambda_n$ is non-standard.
\end{proof}

\end{pr}

\section*{Acknowledgements}
The authors would like to thank Prof. Rainer Schulze-Pillot from Universitaet des Saarlandes for pointing out some valuable information on the topics in this paper, and also to Chunhui Liu from Universit\'e
Paris Diderot - Paris 7 for his helpful advices.





\bibliographystyle{model1a-num-names}
\bibliography{<your-bib-database>}







\end{document}